\documentclass[12pt, reqno]{amsart}%
\usepackage{amsmath, amsthm, amscd, amsfonts, amssymb, graphicx, color}
\usepackage[bookmarksnumbered, colorlinks, plainpages]{hyperref}
\usepackage{amsmath}
\usepackage{amsfonts}
\usepackage{amssymb}
\usepackage{graphicx}%
\providecommand{\U}[1]{\protect\rule{.1in}{.1in}}
\hypersetup{colorlinks=true,linkcolor=red, anchorcolor=green, citecolor=cyan, urlcolor=red, filecolor=magenta, pdftoolbar=true}
\newtheorem{theorem}{Theorem}[section]
\newtheorem{lemma}[theorem]{Lemma}
\newtheorem{proposition}[theorem]{Proposition}
\newtheorem{corollary}[theorem]{Corollary}
\theoremstyle{definition}
\newtheorem{definition}[theorem]{Definition}
\newtheorem{example}[theorem]{Example}

\newtheorem{problem}[theorem]{Problem}
\theoremstyle{remark}

\numberwithin{equation}{section}
\begin{document}
\title[Maximal Tracial Algebras]{Maximal Tracial Algebras}
\author[D. Hadwin \MakeLowercase{and} H. Yousefi]{Don Hadwin$^{1}$ \MakeLowercase{and} Hassan Yousefi$^{2\ast}$}
\address{$^{1}$Department of Mathematics \& Statistics, University of New Hampshire, 33
Academic Way, Durham, NH 03824\\
$^{2}$Department of Mathematics, California State University Fullerton, CA
92831, USA.}
\email{\textcolor[rgb]{0.00,0.00,0.84}{hyousefi@fullerton.edu}}
\date{Received: xxxxxx; Revised: yyyyyy; Accepted: zzzzzz. }
\date{\indent$^{\ast}$Corresponding author}
\date{Received: xxxxxx; Revised: yyyyyy; Accepted: zzzzzz. }
\date{\indent$^{\ast}$Corresponding author}

\begin{abstract}
We introduce and study the notion of \emph{maximal tracial algebras.} We prove
several results in a general setting based on dual pairs and multiplier pairs.
In a special case that $X$ is a Banach space we determine the abelian
subalgebras of $\mathcal{B}\left(  X\right)  $ that are maximal tracial for
rank-one tensors. In another special case that $\mathcal{H\ }$is a Hilbert
space we show that a unital weak-operator closed subalgebra $\mathcal{A}$ of
$\mathcal{B}\left(  \mathcal{H}\right)  $ is abelian and transitive if and
only if it is maximal $e\otimes e$-tracial for every unit vector $e$ in
$\mathcal{H}$. We also make slight connections between our ideas and the
Kadison Similarity Problem and also the Connes' Embedding Problem.

\end{abstract}
\maketitle

\setcounter{page}{1}


\let\thefootnote\relax\footnote{Copyright 2016 by the Tusi Mathematical
Research Group.}

\section{Introduction and Definitions}

Algebras with traces have played an important role in linear algebra and the
theory of operator algebras \cite{Russo}, \cite{Cuntz}, \cite{Kad2}. We
consider a situation where $\mathcal{A}$ is a unital subalgebra of an algebra
$\mathcal{B}$ and $\varphi$ is a functional on $\mathcal{B}$ with
$\varphi\left(  1\right)  =1$ that is tracial on $\mathcal{A}$. We study the
case when $\mathcal{A}$ is a maximal such algebra for $\varphi.$ For the case
when $\mathcal{B}$ is the algebra of $2\times2$ matrices over an arbitrary
field, we completely characterize the pairs $\left(  \mathcal{A}%
,\varphi\right)  $ so that $\mathcal{A}$ is maximal $\varphi$-tracial. In
particular we show some connections between our ideas in this paper and two
very famous problems in Functional Analysis, namely the \emph{Kadison
Similarity Problem} \cite{Kad} and the \emph{Connes' Embedding Problem}
\cite{Connes}.

Throughout this paper $\mathcal{H}$ is a separable Hilbert space over the
field of complex numbers $\mathbb{C}$ and $\mathcal{B}\left(  \mathcal{H}%
\right)  $ denotes the set of all bounded operators on $\mathcal{H}.$ We show
the set of all $k\times k$ matrices over a field $\mathbb{F}$ by $M_{k}\left(
\mathbb{F}\right)  $ and for $A\in M_{k}\left(  \mathbb{F}\right)  $ we use
$Tr\left(  A\right)  $ to denote the usual trace of $A$. The identity matrix
in $M_{k}\left(  \mathbb{F}\right)  $ is shown by $I_{k}$ and by $e_{ij}$
($1\leq i,j\leq k$) we mean the standard matrix units in $M_{k}\left(
\mathbb{F}\right)  $, i.e. the $k\times k$ matrix whose all entries are $0$'s
except the $\left(  i,j\right)  $-entry which is $1$. For $x,y\in\mathcal{H}$,
the rank-one operator $x\otimes y$ is defined by $\left(  x\otimes y\right)
h=\left\langle h,y\right\rangle x$. Note that if $A,B\in\mathcal{B}\left(
\mathcal{H}\right)  $, then $A\left(  x\otimes y\right)  B=Ax\otimes B^{\ast
}y.$ If $\mathcal{B}$ is an algebra over a field $\mathbb{F}$ and
$\mathcal{S}\subseteq\mathcal{B}$, then by the commutant $\mathcal{S}^{\prime
}$ of $\mathcal{S}$ in $\mathcal{B}$ we mean $\mathcal{S}^{\prime}=\left\{
B\in\mathcal{B}:BS=SB,\forall S\in\mathcal{S}\right\}  .$ If the algebra
$\mathcal{B}$ is unital, then we show its unit by $1.$

Section one of this paper contains some definitions, notations, terminologies,
and a motivating example.

Most of our main results are in section two where we prove several theorems in
very general cases for \emph{dual pairs}. Theorem \ref{10} is one of the main
results of this section in which we characterize when an abelian algebra is
maximal $\varphi$-tracial provided that $\varphi$ is a rank-one functional.
This result, for instance, is used in a special case in Corollary \ref{25} to
determine which abelian subalgebras of $\mathcal{B}\left(  X\right)  $ are
maximal tracial for rank-one tensors $x\otimes\alpha,$ where $X$ is a Banach
space, $x\in X,$ and $\alpha$ is a unital linear functional in the normed dual
of $X.$ In this section, we also show that a unital weak-operator closed
subalgebra $\mathcal{A}$ of $\mathcal{B}\left(  \mathcal{H}\right)  $ is
maximal $e\otimes e$-tracial for every unit vector $e$ in $\mathcal{H}$ if and
only if $\mathcal{A}$ is abelian and transitive. We also provide several
interesting examples in this section.

Section three deals with Multiplier Pairs. We prove a general statement in
Theorem \ref{35} about multiplier pairs that gives a lot of examples of
maximal tracial algebras. For instance we use Theorem \ref{35} in section four
where maximal tracial algebras of von Neumann algebras are studied.

In section four, after proving a few theorems, we raise the following
question: If a von Neumann algebra $\mathcal{M}$ on a Hilbert space
$\mathcal{H}$ is $\left(  e\otimes f\right)  $-tracial, where $e,f\in
\mathcal{H}$ are cyclic vectors for $\mathcal{M}$ and $\left\langle
e,f\right\rangle =1$, then should $\mathcal{M}$ be maximal $\left(  e\otimes
f\right)  $-tracial? For the case where $\mathcal{H}$ is separable, we will
reduce the problem to the case where $\mathcal{M}$ is a finite factor von
Neumann algebra, and such a factor has a unique norm continuous unital tracial
functional. We will also show a slight relationship between some of our ideas
in this section and the Kadison Similarity Problem \cite{Kad}.

In the last section of our paper, by introducing a new type of ultraproduct,
we make a connection between some of our results in this paper and the Connes'
Embedding Problem \cite{Connes}. We will show in Theorem \ref{60} that the
analogue of Connes' embedding problem has an affirmative answer in the setting
of maximal tracial ultraproducts.

\begin{definition}
Suppose $\mathcal{A}$ and $\mathcal{B}$ are unital algebras over a field
$\mathbb{F}$ with $1\in\mathcal{A}\subset\mathcal{B}$ and suppose
$\varphi:\mathcal{B}\rightarrow\mathbb{F}$ is a linear map such that
$\varphi\left(  1\right)  =1$. We say that $\mathcal{A}$ is a \emph{tracial
algebra} for $\varphi$ or $\mathcal{A}$ is $\varphi$\emph{-tracial }(or
$\varphi$ is a \emph{tracial functional for }$\mathcal{A}$) if, for every
$x,y\in\mathcal{A}$,%
\[
\varphi\left(  xy\right)  =\varphi\left(  yx\right)  \text{ .}%
\]
We say that $\mathcal{A}$ is a \emph{maximal tracial algebra} for $\varphi$ in
$\mathcal{B}$ if $\mathcal{A}$ is tracial for $\varphi$ and no larger algebra
of $\mathcal{A}$ is tracial for $\varphi$.
\end{definition}

The following is our first simple, yet useful, result. If $\mathcal{B}$ is
unital algebra over a Hausdorff field $\mathbb{F}$, then by an \emph{algebra
topology} on $\mathcal{B}$ we mean a topology for which addition, scalar
multiplication and the multiplication in each variable are all continuous.

\begin{proposition}
Suppose $1\in\mathcal{B}$ is an algebra over a field $\mathbb{F}$,
$\varphi:\mathcal{B}\rightarrow\mathbb{F}$ is linear. Then

\begin{enumerate}
\item If $\mathcal{A}\subset\mathcal{B}$ is a $\varphi$-tracial algebra, then
$\mathcal{A+}\mathbb{F}1$ is also $\varphi$-tracial.

\item If $\left\{  \mathcal{A}_{i}:i\in I\right\}  $ is an increasingly
directed family of $\varphi$-tracial algebras, then so is $\bigcup_{i\in
I}\mathcal{A}_{i}.$

\item Every $\varphi$-tracial algebra is contained in a maximal $\varphi
$-tracial algebra.

\item Every maximal $\varphi$-tracial algebra is unital.

\item If $\mathbb{F}$ is a Hausdorff field, $\mathcal{A}$ is $\varphi
$-tracial, and $\varphi$ is continuous with respect to some Hausdorff algebra
topology $\mathcal{T}$ on $\mathcal{B}$ , then $\mathcal{A}^{-\mathcal{T}}$ is
also $\varphi$-tracial.

\item If $\mathbb{F}$ is a Hausdorff field, $\mathcal{A}$ is $\varphi
$-tracial, and $\varphi$ is continuous with respect to some Hausdorff algebra
topology $\mathcal{T}$ on $\mathcal{B}$, then every maximal $\varphi$-tracial
algebra in $\mathcal{B\ }$is $\mathcal{T}$-closed.
\end{enumerate}
\end{proposition}

\begin{proof}
$\left(  1\right)  .$ This follows from the fact that for every $a,b\in
\mathcal{A}$ and every $\alpha,\beta\in\mathbb{F}$ we have
\[
\left(  a+\alpha1\right)  \left(  b+\beta1\right)  -\left(  b+\beta1\right)
\left(  a+\alpha1\right)  =ab-ba.
\]

$\left(  2\right)  $-$\left(  4\right)  $. These are trivial.

$\left(  5\right)  .$ Suppose $a,b\in\mathcal{A}^{-\mathcal{T}}$. Then there
are nets $\left\{  a_{n}\right\}  ,$ $\left\{  b_{m}\right\}  $ in
$\mathcal{A}$ such that $a_{n}\rightarrow a$ and $b_{m}\rightarrow b$ with
respect to $\mathcal{T}$. Thus, for each $m$,
\[
a_{n}b_{m}\rightarrow ab_{m}\text{ and }b_{m}a_{n}\rightarrow b_{m}a.
\]
Thus%
\[
\varphi\left(  b_{m}a\right)  =\lim_{n}\varphi\left(  b_{m}a_{n}\right)
=\lim_{n}\varphi\left(  a_{n}b_{m}\right)  =\varphi\left(  ab_{m}\right)  .
\]
Similarly,%
\[
\varphi\left(  ba\right)  =\lim_{m}\varphi\left(  b_{m}a\right)  =\lim
_{m}\varphi\left(  ab_{m}\right)  =\varphi\left(  ab\right)
\]

$\left(  6\right)  .$ This easily follows from $\left(  5\right)  $.
\end{proof}

We now look at a simple matrix example which was the motivating example for
this paper.

\begin{example}
\label{matrix}Suppose $\mathbb{F}$ is any field. A linear functional
$\varphi:M_{2}\left(  \mathbb{F}\right)  \rightarrow\mathbb{F}$ is given by a
matrix $K\in M_{2}\left(  \mathbb{F}\right)  $ and is defined by
$\varphi\left(  T\right)  =Tr\left(  TK\right)  ,$ where $Tr$ is the usual
trace. In this case we write $\varphi=\varphi_{K}$. Here
\[
K=\left(
\begin{array}
[c]{cc}%
\varphi\left(  e_{11}\right)  & \varphi\left(  e_{21}\right) \\
\varphi\left(  e_{12}\right)  & \varphi\left(  e_{22}\right)
\end{array}
\right)  ,
\]
where $e_{ij}$ ($1\leq i,j\leq2$) are the standard matrix units for
$M_{2}\left(  \mathbb{F}\right)  $. The condition $\varphi_{K}\left(
1\right)  =1$ implies $Tr\left(  K\right)  =1.$ Also, if $\mathcal{A}$ is a
unital algebra in $M_{2}\left(  \mathbb{F}\right)  $ and $S\in M_{2}\left(
\mathbb{F}\right)  $ is invertible, then $\varphi_{K}$ is tracial for
$S\mathcal{A}S^{-1}$ if and only if $\varphi_{S^{-1}KS}$ is tracial for
$\mathcal{A}$. Hence the problem of finding the tracial unital functionals for
$\mathcal{A}$ is the same as for $S\mathcal{A}S^{-1}$.

For a given unital subalgebra $\mathcal{A}$ of $M_{2}\left(  \mathbb{F}%
\right)  $, we want to find all functionals $\varphi$ for which $\mathcal{A}$
is maximal $\varphi\,$-tracial.

If $\dim\left(  \mathcal{A}\right)  =1$, then $\mathcal{A}$ is abelian, but
not maximal abelian, so $\mathcal{A}$ is not maximal tracial for any functional.

If $\dim\left(  \mathcal{A}\right)  =4$, then the only tracial functional on
$\mathcal{A}$ is $\lambda Tr\left(  {}\right)  $ for some $\lambda$ in
$\mathbb{F}$. If the characteristic of $\mathbb{F}$ is $2$, then $\lambda
Tr\left(  1\right)  =0\neq1$, and no unital $\varphi$ exists, and if the
characteristic of $\mathbb{F}$ is not $2,$ then $\varphi=\frac{1}{2}Tr\left(
{}\right)  $ is the only tracial unital functional on $\mathcal{A}$.

Next, suppose $\dim\left(  \mathcal{A}\right)  =2$. Then there is a matrix
$T\notin\mathbb{F}I_{2}$ such that $\mathcal{A}=\mathbb{F}I_{2}+\mathbb{F}T,$
that is, the algebra of polynomials in $T.$ We want to find the maximal
$\varphi$-tracial algebras in $M_{2}\left(  \mathbb{F}\right)  $. Let
$p\left(  \lambda\right)  =\lambda^{2}-Tr\left(  T\right)  \lambda+\det\left(
T\right)  $ be the characteristic polynomial of $T$.

If $p\left(  T\right)  $ splits, then either $\mathcal{A}$ is similar to
\[
\mathcal{D}_{2}=\left\{  \left(
\begin{array}
[c]{cc}%
a & 0\\
0 & b
\end{array}
\right)  :a,b\in\mathbb{F}\right\}  \text{ or }\mathcal{T}_{2}=\left\{
\left(
\begin{array}
[c]{cc}%
a & b\\
0 & a
\end{array}
\right)  :a,b\in\mathbb{F}\right\}  .
\]
If $\varphi\neq\frac{1}{2}Tr\left(  {}\right)  $, then $\mathcal{D}_{2}$ or
$\mathcal{T}_{2}$ is maximal $\varphi$-tracial only if there is a
$3$-dimensional algebra $\mathcal{D}$ that contains $\mathcal{D}_{2}$ or
$\mathcal{T}_{2}$ for which $\varphi$ is tracial. The only $3$-dimensional
algebras containing $\mathcal{D}_{2}$ are
\[
\mathcal{U}_{2}=\left\{  \left(
\begin{array}
[c]{cc}%
a & b\\
0 & c
\end{array}
\right)  :a,b,c\in\mathbb{F}\right\}  \text{ or }\mathcal{L}_{2}=\left\{
\left(
\begin{array}
[c]{cc}%
a & 0\\
b & c
\end{array}
\right)  :a,b\in\mathbb{F}\right\}  .
\]
Since $e_{11}e_{12}=e_{12}$ and $e_{12}e_{11}=0,$ then a functional $\varphi$
is tracial on $\mathcal{U}_{2}$ if and only if $\varphi\left(  e_{12}\right)
=0.$ Similarly, a functional $\varphi$ is tracial on $\mathcal{L}_{2}$ if and
only if $\varphi\left(  e_{21}\right)  =0.$ Thus $\mathcal{D}_{2}$ is maximal
tracial for $\varphi$ if and only if $\varphi\left(  e_{12}\right)  \neq0$ and
$\varphi\left(  e_{21}\right)  \neq0$. Also the only $3$-dimensional algebra
containing $\mathcal{T}_{2}$ is $\mathcal{U}_{2}$, so $\mathcal{T}_{2}$ is
maximal $\varphi$-tracial if and only if $\varphi\left(  e_{12}\right)
\neq0.$

If $p\left(  \lambda\right)  $ doesn't split, then $\mathcal{A}$ is isomorphic
to $\mathbb{F}\left[  x\right]  /\left\langle p\left(  x\right)  \right\rangle
$, which is a field. If $\mathcal{D}$ is an algebra and $\mathcal{A}%
\varsubsetneq\mathcal{D}$, then%
\[
\dim_{\mathbb{F}}\left(  \mathcal{D}\right)  =2\dim_{\mathcal{A}}\left(
\mathcal{D}\right)  \geq4.
\]
Hence $\mathcal{A}$ is a maximal subalgebra of $M_{2}\left(  \mathbb{F}%
\right)  $. Thus if $\varphi\neq\frac{1}{2}Tr\left(  {}\right)  $, then
$\mathcal{A}$ is maximal $\varphi$-tracial.

If $\dim\left(  \mathcal{A}\right)  =3,$ then the characteristic polynomial of
every $T\in\mathcal{A}$ splits, and $\mathcal{A}$ must be similar to
$\mathcal{U}_{2}$, which is maximal $\varphi$-tracial if and only if
$\varphi\neq\frac{1}{2}Tr\left(  {}\right)  $ and $\varphi\left(
e_{12}\right)  \neq0.$
\end{example}

\section{General Dual Pairs}

Dual pairs over a Hausdorff field were studied in \cite{GenRef}. Suppose that
$\mathbb{F}$ is a topological field with a Hausdorff topology, e.g., a
subfield of the complex numbers with the usual topology, or an arbitrary field
with the discrete topology. Next suppose that $X$ is a vector space over
$\mathbb{F}$ and $Y$ is a vector space of linear maps from $X$ to $\mathbb{F}$
such that $\cap\left\{  \ker f:f\in Y\right\}  =\left\{  0\right\}  $ (that
is, $Y$ separates the points of $X$). Such a pair $(X,Y)$ is called a
\emph{dual pair} over $\mathbb{F}$. In \cite{GenRef} dual pairs were used to
unify some of the results on different versions of reflexivity for algebras
and linear spaces of $\mathcal{B}\left(  \mathcal{H}\right)  $. In this
section, for a dual pair $(X,Y)$, we will prove some theorems for continuous
linear maps on $X$ and $Y$ that will provide several important results in
special cases of dual pairs.

\begin{definition}
Let $X$ be a vector space over a field $\mathbb{F}$, $X^{\prime}$ be the dual
space of $X,$ and let $L\left(  X\right)  $ be the set of linear maps on $X.$
For a linear map $\alpha:X\rightarrow\mathbb{F}$ and $x_{0}\in X$ we define
the linear maps $x_{0}\otimes\alpha:L\left(  X\right)  \rightarrow\mathbb{F}$
and $\alpha\otimes x_{0}:L\left(  X^{\prime}\right)  \rightarrow\mathbb{F}$ by%
\[
\left(  x_{0}\otimes\alpha\right)  \left(  T\right)  =\alpha\left(  T\left(
x_{0}\right)  \right)  \text{ and }\left(  \alpha\otimes x_{0}\right)  \left(
S\right)  =S\left(  \alpha\right)  \left(  x_{0}\right)  \text{.}%
\]
For a dual pair $(X,Y)$ over a Hausdorff field $\mathbb{F}$ we let
$\sigma\left(  X,Y\right)  $-topology on $X$ be the smallest topology on $X$
that makes all of the maps in $Y$ continuous on $X$, and we let the
$\sigma\left(  Y,X\right)  $-topology on $Y$ be the smallest topology for
which the map $f\rightarrow f\left(  x\right)  $ is continuous on $Y$ for each
$x$ in $X$. For $A\subseteq X$ and $B\subseteq Y$, we define
\[
A^{\bot}=\left\{  f\in Y:f|_{A}=0\right\}  \text{ and }B_{\bot}=\cap\left\{
\ker\left(  f\right)  :f\in B\right\}  .
\]
For a subset $D$ of either $X$ or $Y$, we denote the linear span of $D$ by
$sp(D).$
\end{definition}

It was shown in \cite[Proposition 1.1]{GenRef} that
\[
\left(  A^{\perp}\right)  _{\perp}=sp\left(  A\right)  ^{-\sigma\left(
X,Y\right)  }\text{ and }\left(  B_{\perp}\right)  ^{\perp}=sp\left(
B\right)  ^{-\sigma\left(  Y,X\right)  }%
\]
always hold. In particular, if $0\neq x\in X$, there is an $f\in Y$ such that
$f\left(  x\right)  =1$. It was also shown in \cite[Proposition 1.1]{GenRef}
that if $f:X\rightarrow\mathbb{F}$ is linear and $\sigma\left(  X,Y\right)  $
continuous, then $f\in Y$.

Suppose $\left(  X,Y\right)  $ is a dual pair over a Hausdorff field
$\mathbb{F}$. Let $L_{\sigma\left(  X,Y\right)  }\left(  X\right)  $ denote
the linear transformations $T:X\rightarrow X$ that are $\sigma\left(
X,Y\right)  $-$\sigma\left(  X,Y\right)  $ continuous. We give $L\left(
X\right)  $ the weak operator topology (WOT) defined as the topology of
pointwise $\sigma\left(  X,Y\right)  $ convergence. This means that the weak
operator topology is the weak topology on $L_{\sigma\left(  X,Y\right)
}\left(  X\right)  $ induced by the set of linear functionals $\mathcal{E}%
=\left\{  x\otimes y:x\in X,y\in Y\right\}  $. Consequently, if $\mathcal{W}%
=sp\left(  \mathcal{E}\right)  $, then $\left(  L_{\sigma\left(  X,Y\right)
}\left(  X\right)  ,\mathcal{W}\right)  $ is a dual pair and the weak operator
topology is the $\sigma\left(  L_{\sigma\left(  X,Y\right)  }\left(  X\right)
,\mathcal{W}\right)  $ topology. Thus, by \cite[Proposition 1.1]{GenRef}, a
linear map $\varphi:L_{\sigma\left(  X,Y\right)  }\left(  X\right)
\rightarrow\mathbb{F}$ is WOT continuous if and only if it is in $\mathcal{W}%
$, i.e., of the form $\varphi=\sum_{k=1}^{n}x_{k}\otimes\alpha_{k}$ for
$x_{1},\ldots,x_{n}\in X$ and $\alpha_{1},\ldots,\alpha_{n}\in Y$.

If we denote the duality of $X$ and $Y$ with the notation%
\[
\alpha\left(  x\right)  =\left\langle x,\alpha\right\rangle
\]
for $x\in X$ and $\alpha\in Y$, then, for every $T\in L_{\sigma\left(
X,Y\right)  }\left(  X\right)  $, there is a $T^{\#}\in L_{\sigma\left(
Y,X\right)  }\left(  Y\right)  $ such that, for every $x\in X$ and every
$\alpha\in Y$,%
\[
\left\langle Tx,\alpha\right\rangle =\left\langle x,T^{\#}\alpha\right\rangle
\text{.}%
\]
Clearly, $T^{\#\#}=T,$ $1^{\#}=1,$ $\left(  ST\right)  ^{\#}=T^{\#}S^{\#},$
$\left(  \lambda S+T\right)  ^{\#}=\lambda S^{\#}+T^{\#}$. If $\mathcal{S}%
\subset L_{\sigma\left(  X,Y\right)  }\left(  X\right)  $, we let
$\mathcal{S}^{\#}$ denote $\left\{  S^{\#}:S\in\mathcal{S}\right\}  $. If
$\varphi$ is a WOT continuous linear functional on $L_{\sigma\left(
X,Y\right)  }\left(  X\right)  $, we define the adjoint functional
$\varphi^{\#}$ on $L_{\sigma\left(  Y,X\right)  }\left(  Y\right)  $ by%
\[
\varphi^{\#}\left(  T^{\#}\right)  =\varphi\left(  T\right)  .
\]
Clearly $\left(  x\otimes\alpha\right)  ^{\#}=\alpha\otimes x$.

The following lemma can be proved easily.

\begin{lemma}
\label{5}Suppose that $(X,Y)$ is a dual pair, $\mathcal{A}$ is a unital
subalgebra of $L_{\sigma\left(  X,Y\right)  }\left(  X\right)  ,$ and that
$\varphi$ is a WOT continuous linear functional such that $\varphi\left(
1\right)  =1$. Then $\mathcal{A}$ is $\varphi$-tracial if and only if
$\mathcal{A}^{\#}$ is $\varphi^{\#}$-tracial and $\mathcal{A}$ is maximal
$\varphi$-tracial if and only if $\mathcal{A}^{\#}$ is maximal $\varphi^{\#}$-tracial.
\end{lemma}

Note that every abelian algebra is $\varphi$-tracial for every $\varphi$. The
next theorem characterizes when an abelian algebra is maximal $\varphi
$-tracial provided $\varphi$ is a rank-one functional.

\begin{theorem}
\label{10}Suppose that $(X,Y)$ is a dual pair, $\mathcal{A}$ is an abelian
unital subalgebra of $L_{\sigma\left(  X,Y\right)  }\left(  X\right)  $, and
that $x\in X,$ $\alpha\in Y$ with $\left(  x\otimes\alpha\right)  \left(
1\right)  =\alpha\left(  x\right)  =1.$ Then $\mathcal{A}$ is maximal tracial
with respect to $x\otimes\alpha$ if and only if the following conditions hold:

\begin{enumerate}
\item $\mathcal{A}$ is maximal abelian in $L_{\sigma\left(  X,Y\right)
}\left(  X\right)  ,$

\item $\left[  \mathcal{A}x\right]  ^{-\sigma\left(  X,Y\right)  }=X,$ and

\item $\left[  \mathcal{A}^{\#}\alpha\right]  ^{-\sigma\left(  Y,X\right)
}=Y$.
\end{enumerate}
\end{theorem}

\begin{proof}
Suppose $\mathcal{A}$ is maximal tracial with respect to $x\otimes\alpha$.
Clearly, $\mathcal{A}$ is maximal abelian. Assume, via contradiction, that
$M=\left[  \mathcal{A}x\right]  ^{-\sigma\left(  X,Y\right)  }\neq X$. Let
\[
\mathcal{A}_{1}=\left\{  T\in L_{\sigma\left(  X,Y\right)  }\left(  X\right)
:T\left(  M\right)  \subset M\text{, }T|_{M}\in\mathcal{A}|_{M}\right\}  .
\]
Then $\mathcal{A}_{1}$ is tracial with respect to $x_{0}\otimes\alpha$ and,
since $M\neq X$, $\mathcal{A}_{1}$ is not abelian. This contradiction implies
$\left[  \mathcal{A}x\right]  ^{-\sigma\left(  X,Y\right)  }=X$.

From Lemma \ref{5} we know that if $\mathcal{A}$ is maximal $\left(
x\otimes\alpha\right)  $-tracial, then $\mathcal{A}^{\#}$ is maximal $\left(
x\otimes\alpha\right)  ^{\#}$-tracial. Since $\left(  x\otimes\alpha\right)
^{\#}=\alpha\otimes x$, we conclude $\left[  \mathcal{A}^{\#}\alpha\right]
^{-\sigma\left(  Y,X\right)  }=Y$.

Now suppose $\left(  1\right)  $-$\left(  3\right)  $ hold. Suppose $T\in
L_{\sigma\left(  X,Y\right)  }\left(  X\right)  $ and the algebra generated by
$\mathcal{A}\cup\left\{  T\right\}  $ is tracial for $x\otimes\alpha$. Then,
for every $A,B,C\in\mathcal{A}$, we have
\begin{align*}
\left(  x\otimes\alpha\right)  \left(  B\left(  AT-TA\right)  C\right)   &
=\left(  x\otimes\alpha\right)  \left(  \left(  AT-TA\right)  CB\right) \\
&  =\left(  x\otimes\alpha\right)  \left(  AT\left(  CB\right)  -T\left(
CB\right)  A\right) \\
&  =\left(  x\otimes\alpha\right)  \left(  \left(  CB\right)  AT-T\left(
CB\right)  A\right) \\
&  =0.
\end{align*}
This means%
\[
\left[  AT-TA\right]  \left(  Cx\right)  \in\ker\left(  B^{\prime}%
\alpha\right)
\]
for every $B,C\in\mathcal{A}$. It follows from $\left(  2\right)  $ and
$\left(  3\right)  $ that $AT=TA$. Hence, by $\left(  1\right)  $,
$T\in\mathcal{A}$.
\end{proof}

A subalgebra $\mathcal{A}$ of $L_{\sigma\left(  X,Y\right)  }\left(  X\right)
$ is called \emph{transitive} if the only $\sigma\left(  X,Y\right)  $-closed
$\mathcal{A}$-invariant linear subspaces are $\left\{  0\right\}  $ and $X$.
It is easily shown that $M$ is $\mathcal{A}$-invariant if and only if
$M^{\perp}$ is $\mathcal{A}^{\#}$-invariant. Thus $\mathcal{A}$ is transitive
in $L_{\sigma\left(  X,Y\right)  }\left(  X\right)  $ if and only if
$\mathcal{A}^{\#}$ is transitive in $L_{\sigma\left(  Y,X\right)  }\left(
Y\right)  $.

\begin{theorem}
\label{15}A subalgebra $\mathcal{A}$ of $L_{\sigma\left(  X,Y\right)  }\left(
X\right)  $ is maximal abelian and transitive if and only if $\mathcal{A}$ is
maximal $x\otimes\alpha$-tracial for every $x\in X$ and $\alpha\in Y$ with
$\alpha\left(  x\right)  =\left(  x\otimes\alpha\right)  \left(  1\right)  =1$.
\end{theorem}

\begin{proof}
Suppose $\mathcal{A}$ is maximal $x\otimes\alpha$-tracial for every $x\in X$
and $\alpha\in Y$ with $\alpha\left(  x\right)  =\left(  x\otimes
\alpha\right)  \left(  1\right)  =1$. Let $S,T\in\mathcal{A}$. Assume, via
contradiction, that $ST\neq TS.$ Then there is an $x\in X$ such that $\left(
ST-TS\right)  x\neq0$. We know that there are $\alpha_{1},\alpha_{2}\in Y$
such that $\alpha_{1}\left(  \left(  ST-TS\right)  x\right)  \neq0$ and
$\alpha_{2}\left(  x\right)  \neq0$. Thus there is an $\alpha\in\left\{
\alpha_{1},\alpha_{2},\alpha_{1}+\alpha_{2}\right\}  $ such that
$\alpha\left(  x\right)  \neq0$ and $\alpha\left(  \left(  ST-TS\right)
x\right)  \neq0$. By dividing by $\alpha\left(  x\right)  $ we can assume that
$\alpha\left(  x\right)  =1$ and thus
\[
\left(  x\otimes\alpha\right)  \left(  ST-TS\right)  =\alpha\left(  \left(
ST-TS\right)  x\right)  \neq0,
\]
a contradiction. Thus $\mathcal{A}$ is abelian. Since $\mathcal{A}$ is maximal
tracial for every $x\otimes\alpha$ with $\alpha\left(  x\right)  =1,$ it
follows that $\mathcal{A}$ is maximal abelian.

Next suppose $x\in X$ and $x\neq0.$ Then there is an $\alpha\in Y$ such that
$\alpha\left(  x\right)  =1.$ Since $\mathcal{A}$ is abelian and maximal
$\left(  x\otimes\alpha\right)  $-tracial, we know from Theorem \ref{10} that
$\left[  \mathcal{A}x\right]  ^{-\sigma\left(  X,Y\right)  }=X$. Thus
$\mathcal{A}$ is transitive.

The proof of the other direction is straightforward by using Theorem \ref{10}.
\end{proof}

\begin{corollary}
\label{20}Suppose $\mathcal{H}$ is a Hilbert space and $\mathcal{A}\subset
B\left(  \mathcal{H}\right)  $ is a unital weak-operator closed algebra. Then
$\mathcal{A}$ is maximal $e\otimes e$-tracial for every unit vector $e$ in
$\mathcal{H}$ if and only if $\mathcal{A}$ is abelian and transitive.
\end{corollary}

\begin{proof}
Following the proof of Theorem \ref{15} , we just need to show $\mathcal{A}$
is abelian. If $S,T\in\mathcal{A}$ and $ST-TS\neq0,$ then the numerical range
of $ST-TS$ is not $\left\{  0\right\}  $, which means that there is a unit
vector $e\in\mathcal{H}$ such that%
\[
\left(  e\otimes e\right)  \left(  ST-TS\right)  =\left\langle \left(
ST-TS\right)  e,e\right\rangle \neq0,
\]
which contradicts the assumption that $\mathcal{A}$ is $\left(  e\otimes
e\right)  $-tracial. Hence $\mathcal{A}$ is abelian and the rest follows as in
the proof of Theorem \ref{15} .
\end{proof}

We need to interpret the preceding results for Banach spaces. Suppose $X$ is a
Banach space and $Y$ is the dual space of $X.$ Then the $\sigma\left(
X,Y\right)  $-topology on $X$ is the weak topology and the $\sigma\left(
Y,X\right)  $-topology on $Y$ is the weak*-topology. Hence $L_{\sigma\left(
X,Y\right)  }\left(  X\right)  $ is the set of all linear transformations $T$
on $X$ that are weak-weak continuous, which, by the closed graph theorem,
simply means that $L_{\sigma\left(  X,Y\right)  }\left(  X\right)  $ is the
set of all linear transformations $T$ on $X$ that are bounded on $X$. However,
$L_{\sigma\left(  Y,X\right)  }\left(  Y\right)  $ is the set of all linear
transformations on $Y$ that are weak*-weak* continuous, which are precisely
the adjoints of transformations in $B\left(  X\right)  $. Thus $L_{\sigma
\left(  Y,X\right)  }\left(  Y\right)  \neq B\left(  Y\right)  ,$ except when
$X$ is a reflexive Banach space. If $T\in L_{\sigma\left(  X,Y\right)
}\left(  X\right)  $, then $T^{\#}$ is the usual Banach-space adjoint of $T$.
However, if $S\in L_{\sigma\left(  Y,X\right)  }\left(  Y\right)  $, then
$S=T^{\#}$ for some $T\in B\left(  X\right)  $, and our definition of $S^{\#}$
is $S^{\#}=T$. Nonetheless, the Banach-space adjoint of $S$ acts on normed
dual of $Y,$ which is the second dual of $X$. For example, C. Read \cite{Read}
has constructed a transitive operator $T$ on $\ell^{1}$. But no operator $A$
is transitive on the dual $\ell^{\infty}$ of $\ell^{1}$, since, for any
nonzero $x\in\ell^{\infty}$ the closed linear span $M$ of $\left\{
x,Ax,\ldots\right\}  $ is nonzero, and not $\ell^{\infty}$ ($M$ is separable)
and $A$-invariant. However, $T^{\#}$ has no nontrivial weak*-closed invariant
subspaces of $\ell^{\infty}$.

\begin{corollary}
\label{25}Suppose $X$ is a Banach space with normed dual $Y$, $\mathcal{A}%
\subset B\left(  X\right)  $ is an abelian algebra, and $x\in X,$ $\alpha\in
Y$ with $\alpha\left(  x\right)  =1$. Then a maximal abelian algebra
$\mathcal{A}\subset B\left(  X\right)  $ is a maximal $x\otimes\alpha$-tracial
algebra if and only if $\left[  \mathcal{A}x\right]  ^{-\left\Vert
.\right\Vert }=X$ and $\left[  \mathcal{A}^{\#}\alpha\right]  ^{-\text{weak*}%
}=Y$.
\end{corollary}

\begin{example}
Suppose $\left(  \Omega,\Sigma,\mu\right)  $ is a probability space and
$\alpha$ is a normalized gauge norm on $L^{\infty}\left(  \mu\right)  ,$ i.e.,
$\alpha$ is a norm such that $\alpha\left(  1\right)  =1,$ and for every $f\in
L^{\infty}\left(  \mu\right)  ,$ $\alpha\left(  f\right)  =\alpha\left(
\left\vert f\right\vert \right)  $. We can extend the domain of $\alpha$ to
every measurable function $f$ by%
\[
\alpha\left(  f\right)  =\sup\left\{  \alpha\left(  h\right)  :h\in L^{\infty
}\left(  \mu\right)  ,\left\vert h\right\vert \leq\left\vert f\right\vert
\text{ a.e. }\left(  \mu\right)  \right\}  .
\]
Then $\mathcal{L}^{\alpha}\left(  \mu\right)  =\left\{  f:\alpha\left(
f\right)  <\infty\right\}  $ is a Banach space and $L^{\alpha}\left(
\mu\right)  $ is defined as the $\alpha$-closure of $L^{\infty}\left(
\mu\right)  $. We say that the gauge norm $\alpha$ is \emph{continuous} if
\[
\lim_{\mu\left(  E\right)  \rightarrow0}\alpha\left(  \chi_{E}\right)  =0.
\]
It is well-known that if $f\in L^{\alpha}\left(  \mu\right)  $ and $h\in
L^{\infty}\left(  \mu\right)  ,$ then%
\[
\alpha\left(  hf\right)  \leq\alpha\left(  f\right)  \left\Vert h\right\Vert
_{\infty}.
\]
Thus $L^{\infty}\left(  \mu\right)  $ acts, via multiplication, as operators
on $L^{\alpha}\left(  \mu\right)  $. It is shown that $L^{\infty}\left(
\mu\right)  $ is actually a maximal abelian algebra of operators on
$L^{\alpha}\left(  \mu\right)  $. The norm $\alpha$ has a \emph{dual norm}
$\alpha^{\prime}$ on $L^{\infty}\left(  \mu\right)  $ defined by%
\[
\alpha^{\prime}\left(  f\right)  =\sup\left\{  \left\Vert fh\right\Vert
_{1}:h\in L^{\infty}\left(  \mu\right)  ,\alpha\left(  h\right)
\leq1\right\}  \text{ .}%
\]
If $\left\Vert .\right\Vert _{1}\leq\alpha$, then $\alpha^{\prime}$ is a
normalized gauge norm, and, for $f\in L^{\alpha}\left(  \mu\right)  $ and
$h\in\mathcal{L}^{\alpha^{\prime}}\left(  \mu\right)  ,$
\[
\left\Vert fh\right\Vert _{1}\leq\alpha\left(  f\right)  \alpha^{\prime
}\left(  h\right)  \text{.}%
\]
If, in addition, $\alpha$ is continuous, then the normed dual of $L^{\alpha
}\left(  \mu\right)  $ is $\mathcal{L}^{\alpha^{\prime}}\left(  \mu\right)  $,
given for $h\in\mathcal{L}^{\alpha^{\prime}}\left(  \mu\right)  ,$ the
functional $\psi_{h}$ on $L^{\alpha}\left(  \mu\right)  $ defined by
\[
\psi_{h}\left(  f\right)  =\int_{\Omega}fhd\mu\text{ .}%
\]
Suppose $f\in L^{\alpha}\left(  \mu\right)  $ and $h\in\mathcal{L}%
^{\alpha^{\prime}}\left(  \mu\right)  $, then $L^{\infty}\left(  \mu\right)  $
is maximal $f\otimes h$-tracial if and only if both $L^{\infty}\left(
\mu\right)  f$ is dense in $L^{\alpha}\left(  \mu\right)  $ and $L^{\infty
}\left(  \mu\right)  h$ is weak*-dense in $\mathcal{L}^{\alpha^{\prime}%
}\left(  \mu\right)  $. This is equivalent to $f\left(  \omega\right)  \neq0$
a.e. $\left(  \mu\right)  $ and $h\left(  \omega\right)  \neq0$ a.e. $\left(
\mu\right)  $. To see this, suppose $f\neq0$ a.e. $\left(  \mu\right)  $.
Suppose $\psi\in L^{\alpha}\left(  \mu\right)  ^{\#}$ and $\psi|_{L^{\infty
}\left(  \mu\right)  f}=0.$ Then there is an $h\in\mathcal{L}^{\alpha^{\prime
}}\left(  \mu\right)  $ such that, $\psi=\psi_{h}.$ Hence, for every $u\in
L^{\infty}\left(  \mu\right)  $,
\[
\int_{\Omega}ufhd\mu=0.
\]
Since $fh\in L^{1}\left(  \mu\right)  $, this means $fh=0$ a.e. $\left(
\mu\right)  ,$ which implies $h=0$ a.e. $\left(  \mu\right)  $. Thus $\psi=0.$
A similar proof shows that if $h\in\mathcal{L}^{\alpha^{\prime}}\left(
\mu\right)  $, then $L^{\infty}\left(  \mu\right)  h\ $is weak*-dense in
$\mathcal{L}^{\alpha}\left(  \mu\right)  $.
\end{example}

\begin{example}
Let $H^{2}$ be the classical Hardy space on the unit disk. It is known that
$H^{\infty}$ (acting as multiplications) is a maximal abelian algebra in
$B\left(  H^{2}\right)  $. The cyclic vectors for $H^{\infty}$ are the outer
functions (see \cite{Helson}). There are many cyclic vectors in $H^{2}$ for
$\left(  H^{\infty}\right)  ^{\#}$. We see that $H^{\infty}$ is $f\otimes
h$-maximal tracial if and only if $f$ is outer and $h$ is a cyclic vector for
$\left(  H^{\infty}\right)  ^{\#}$. Since $1$ is not cyclic for $\left(
H^{\infty}\right)  ^{\#}$, we see that $H^{\infty}$ is \textbf{not}
$1\otimes1$-maximal tracial.
\end{example}

\begin{example}
In the preceding example, $H^{\infty}$ is the unital weakly closed algebra
generated by the unilateral shift operator on $\ell^{2}$. Suppose
$X\in\left\{  c_{0}\right\}  \cup\left\{  \ell^{p}:1\leq p<\infty\right\}  $,
and let
\[
e_{0}=\left(  1,0,\ldots\right)  ,e_{1}=\left(  0,1,0,\ldots\right)
,\ldots\text{ .}%
\]
Suppose $\left\{  \beta_{n}\right\}  _{n\geq0}$ is a bounded sequence of
positive numbers, and define $Te_{n}=\beta_{n}e_{n+1},$ i.e.%
\[
T\left(  a_{0},a_{1},a_{2},\ldots\right)  =\left(  0,\beta_{0}a_{0},\beta
_{1}a_{1},\ldots\right)  .
\]
Then $T$ is a bounded linear operator and $\mathcal{A}=\left\{  T\right\}
^{\prime}=\left\{  p\left(  T\right)  :p\in\mathbb{C}\left[  z\right]
\right\}  ^{-WOT}$ is a maximal abelian algebra in $B\left(  X\right)  $.
Clearly, $e_{0}$ is a cyclic vector for $T$. There are also many weak*-cyclic
vectors for $T^{\#}$ of the form $h=\left(  1,d_{1},d_{2},\ldots\right)  $.
Thus $\mathcal{A}$ is maximal $e_{0}\otimes h$-tracial for many choices of $h$.
\end{example}

\begin{example}
Suppose $I$ is an uncountable set. Then $\ell^{\infty}\left(  I\right)  $ is a
maximal abelian algebra of operators on $\ell^{2}\left(  I\right)  $, but,
since any vector in $\ell^{2}\left(  I\right)  $ has countable support,
$\ell^{\infty}\left(  I\right)  $ has no cyclic vector. Hence $\ell^{\infty
}\left(  I\right)  $ is not maximal tracial for $x\otimes y$ for any
$x,y\in\ell^{2}\left(  I\right)  $, or for any weak*-continuous linear
$\varphi$ functional on $B\left(  \ell^{2}\left(  I\right)  \right)  $ since
such a $\varphi$ involves only countably many vectors.
\end{example}

\begin{definition}
Suppose $\left(  X,Y\right)  $ is a dual pair over a Hausdorff field
$\mathbb{F}$, and $\mathcal{A}$ is a WOT-closed unital subalgebra of
$L_{\sigma\left(  X,Y\right)  }\left(  X\right)  $. We say that $\mathcal{A}$
has a \emph{complemented invariant subspace lattice} if and only if, for every
$\mathcal{A}$-invariant $\sigma\left(  X,Y\right)  $-closed linear subspace
$M$ of $X$, there is an idempotent $P\in L_{\sigma\left(  X,Y\right)  }\left(
X\right)  $ such that $P\left(  X\right)  =M$ and $P\in\mathcal{A}^{\prime}$,
where $\mathcal{A}^{\prime}$ denotes the commutant of $\mathcal{A}.$
\end{definition}

\begin{theorem}
\label{30}Suppose $\left(  X,Y\right)  $ is a dual pair over a Hausdorff field
$\mathbb{F}$, and $\mathcal{A}$ is a WOT-closed unital subalgebra of
$L_{\sigma\left(  X,Y\right)  }\left(  X\right)  $ with a complemented
invariant subspace lattice. If $\mathcal{A}$ is maximal tracial for $e\otimes
f,$ with $e\in X,$ $f\in Y$ and $f\left(  e\right)  =1$, then $\left[
\mathcal{A}e\right]  ^{-\sigma\left(  X,Y\right)  }=X$ and $\left[
\mathcal{A}^{\#}f\right]  ^{-\sigma\left(  Y,X\right)  }=Y$.
\end{theorem}

\begin{proof}
Let $M=\left[  \mathcal{A}e\right]  ^{-\sigma\left(  X,Y\right)  }$. Then $M$
is $\mathcal{A}$-invariant. Thus there is an idempotent $P\in L_{\sigma\left(
X,Y\right)  }\left(  X\right)  $ such that $P\left(  X\right)  =M$ and
$P\in\mathcal{A}^{\prime}$. With respect to the decomposition $X=M\oplus\ker
P$, we can write every $T\in B\left(  X\right)  $ as an operator matrix
\[
T=\left(
\begin{array}
[c]{cc}%
B & C\\
D & E
\end{array}
\right)  .
\]
If $A\in\mathcal{A}$ we have%
\[
A=\left(
\begin{array}
[c]{cc}%
A_{11} & 0\\
0 & A_{22}%
\end{array}
\right)  .
\]
Let $\mathcal{S}$ be the set of all the operators of the form $\left(
\begin{array}
[c]{cc}%
0 & C\\
0 & 0
\end{array}
\right)  $ and let $\mathcal{D}$ $=\mathcal{A+S}$. Since $e\in M,$ then it
follows that $\mathcal{S}$ is a two-sided ideal in $\mathcal{D}$ and that
$\mathcal{S}\subset\ker\left(  e\otimes f\right)  $. Thus for every
$A,B\in\mathcal{A}$ and every $S_{1},S_{2}\in\mathcal{S}$ we have%
\[
\left(  A+S_{1}\right)  \left(  B+S_{2}\right)  -\left(  B+S_{2}\right)
\left(  A+S_{1}\right)  -\left(  AB-BA\right)  \in\mathcal{S}\text{.}%
\]
Hence $\mathcal{D}$ is $e\otimes f$-tracial. Since $\mathcal{A}$ is maximal
$e\otimes f$-tracial, then $P=1$ and $M=X$.

Similarly, suppose $N=\left[  \mathcal{A}^{\#}f\right]  ^{-\sigma\left(
Y,X\right)  }.$ Then $N_{\perp}$ is $\mathcal{A}$-invariant. Thus there is an
idempotent $P\in L_{\sigma\left(  X,Y\right)  }\left(  X\right)
\cap\mathcal{A}^{\prime}$ such that $P\left(  X\right)  =N_{\perp}\subset\ker
f$. Consequently $\left(  1-P^{\#}\right)  \left(  Y\right)  =N$. Writing the
matrix representations of the elements of $L_{\sigma\left(  X,Y\right)
}\left(  X\right)  $ and $\mathcal{A}$ as above, we can let $\mathcal{T}$ be
the set of all operators of the form $\left(
\begin{array}
[c]{cc}%
0 & 0\\
D & 0
\end{array}
\right)  .$ Since the range of any operator in $\mathcal{T}$ is contained in
$N_{\perp}\subset\ker f$, we see that $\mathcal{T}\subset\ker\left(  e\otimes
f\right)  $. Therefore $\mathcal{E=A}+T$ is $e\otimes f$-tracial, and, since
$\mathcal{A}$ is maximal $\left(  e\otimes f\right)  $-tracial, we see that
$N_{\perp}=\left\{  0\right\}  $, and $N=\left(  N_{\perp}\right)  ^{\perp}=Y$.
\end{proof}

\section{Multiplier Pairs}

Multiplier pairs were studied in \cite{HN}. A \emph{multiplier pair} is a pair
$\left(  X,Y\right)  $ where $X$ is a Banach space, $Y$ is a Hausdorff
topological vector space, $X\subset Y$, and there is a multiplication on $X$
with values in $Y$, i.e., a bilinear map%
\[
\cdot:X\times X\rightarrow Y
\]
such that

\begin{enumerate}
\item $\cdot$ is continuous in each variable

\item There is an \emph{identity element} $e\in X$ such that, for every $a\in
X,$%
\[
a\cdot e=e\cdot a=a
\]

\item The sets $\mathcal{L}_{0}=\left\{  a\in X:a\cdot X\subset X\right\}  $
and $\mathcal{R}_{0}=\left\{  b\in X:X\cdot b\subset X\right\}  $ are dense in
$X$,

\item There are dense subsets $E\subset\mathcal{L}_{0}$, $F\subset X$,
$G\subset\mathcal{R}_{0}$, such that, for all $a\in E,b\in F,c\in G$, we have%
\[
\left(  a\cdot b\right)  \cdot c=a\cdot\left(  b\cdot c\right)  .
\]

\end{enumerate}

If $a\in\mathcal{L}_{0}$ and $b\in\mathcal{R}_{0}$, the operators $L_{a}$ and
$R_{b}$ defined on $X$ by%
\[
L_{a}x=a\cdot x\text{ and }R_{b}x=x\cdot b
\]
are bounded on $X$ (see \cite[Theorem 1]{HN}). Moreover, it was shown in
\cite[Theorem 1]{HN} that if $\mathcal{L}=\left\{  L_{a}:a\in\mathcal{L}%
_{0}\right\}  $ and $\mathcal{R=}\left\{  R_{b}:b\in\mathcal{R}_{0}\right\}
$, then $\mathcal{L},\mathcal{R}\subset B\left(  X\right)  $ and%
\[
\mathcal{L}^{\prime}=\mathcal{R}\text{ and }\mathcal{R}^{\prime}%
=\mathcal{L}\text{ .}%
\]
In particular, if the multiplication is commutative, then $\mathcal{L}%
=\mathcal{R}$ is maximal abelian in $B\left(  X\right)  $. In this case the
next theorem reduces to Theorem \ref{10}.

We now prove a result about multiplier pairs that gives a lot of examples of
maximal tracial algebras. In many of the examples in \cite{HN}, condition
$\left(  1\right)  $ in the following theorem holds.

\begin{theorem}
\label{35}Suppose $\left(  X,Y\right)  $ is a multiplier pair with identity
element $e$, and suppose $\alpha\in X^{\#}$ and $\alpha\left(  e\right)  =1$.
Also suppose

\begin{enumerate}
\item $\left\{  R_{B}:B\in\mathcal{L}_{0}\cap\mathcal{R}_{0}\right\}  $ is
WOT-dense in $\mathcal{R}$

\item $\left[  \mathcal{L}^{\#}\alpha\right]  ^{-\text{\textrm{weak*}}}%
=X^{\#}$, and

\item $\mathcal{L}$ is $\left(  e\otimes\alpha\right)  $-tracial
\end{enumerate}

Then $\mathcal{L}$ is maximal $\left(  e\otimes\alpha\right)  $-tracial.
\end{theorem}

\begin{proof}
Suppose $B\in\mathcal{L}_{0}\cap\mathcal{R}_{0}$.

\textbf{Claim:} $L_{B}^{\#}\alpha=R_{B}^{\#}\alpha.$

Proof of Claim: Suppose $A\in\mathcal{L}_{0}$. Then%
\begin{align*}
\left(  L_{B}^{\#}\alpha\right)  \left(  A\right)   &  =\alpha\left(
BA\right)  =\left(  e\otimes\alpha\right)  \left(  L_{B}L_{A}\right) \\
&  =\left(  e\otimes\alpha\right)  \left(  L_{A}L_{B}\right)  =\alpha\left(
AB\right) \\
&  =\left(  R_{B}^{\#}\alpha\right)  \left(  A\right)  .
\end{align*}
Since $\mathcal{L}_{0}$ is dense in $X$, the claim is proved.

Suppose $\mathcal{D}\subset B\left(  X\right)  $ is a $\left(  e\otimes
\alpha\right)  $-tracial algebra containing $\mathcal{L}$, and suppose
$T\in\mathcal{D}$. Choose $B\in\mathcal{L}_{0}\cap\mathcal{R}_{0}$ and
$C,D\in\mathcal{L}_{0}$. We then have%
\begin{align*}
\left\langle \left(  TR_{B}\right)  \left(  C\right)  ,L_{D}^{\#}%
\alpha\right\rangle  &  =\left\langle T\left(  CB\right)  ,L_{D}^{\#}%
\alpha\right\rangle =\left\langle TL_{C}L_{B}e,L_{D}^{\#}\alpha\right\rangle
\\
&  =\left\langle L_{D}TL_{C}L_{B}e,\alpha\right\rangle =\left(  e\otimes
\alpha\right)  \left(  L_{D}TL_{C}L_{B}\right) \\
&  =\left(  e\otimes\alpha\right)  \left(  L_{B}L_{D}TL_{C}\right)
=\left\langle L_{B}L_{D}TL_{C}e,\alpha\right\rangle \\
&  =\left\langle L_{D}TL_{C}e,L_{B}^{\#}\alpha\right\rangle =\left\langle
L_{D}TL_{C}e,R_{B}^{\#}\alpha\right\rangle \\
&  =\left\langle R_{B}L_{D}TL_{C}e,\alpha\right\rangle =\left\langle
L_{D}R_{B}TL_{C}e,\alpha\right\rangle \\
&  =\left\langle \left(  R_{B}T\right)  \left(  C\right)  ,L_{D}^{\#}%
\alpha\right\rangle .
\end{align*}
Since $\mathcal{L}^{\#}\alpha$ is weak* dense in $X^{\#}$ and $\mathcal{L}%
_{0}$ is dense in $X$, we see that
\[
TR_{B}=R_{B}T
\]
for every $B\in\mathcal{L}_{0}\cap\mathcal{R}_{0}$. It follows from $\left(
1\right)  $ that $T\in\mathcal{R}^{\prime}=\mathcal{L}$. Thus $\mathcal{L}$ is
maximal $\left(  e\otimes\alpha\right)  $-tracial.
\end{proof}

\section{von Neumann Algebras}

Suppose $\mathcal{M}$ is a von Neumann algebra acting on a Hilbert space $H$
with a faithful normal tracial state $\tau$. There is a notion of
\emph{convergence in measure }on $\mathcal{M}$ (see \cite{Helson}), defined by
saying a net $\left\{  T_{\lambda}\right\}  $ in $\mathcal{M}$ converges to
$0$ in measure, denoted by $T_{\lambda}\rightarrow0$ $\left(  \tau\right)  $,
if and only if there is a net $\left\{  P_{\lambda}\right\}  $ of projections
in $\mathcal{M}$ such that $\tau\left(  P_{\lambda}\right)  \rightarrow0$ and
$\left\Vert T_{\lambda}\left(  1-P_{\lambda}\right)  \right\Vert
\rightarrow0.$ The completion of $\mathcal{M}$, denoted by $\mathcal{\hat{M}}%
$, with respect to this topology, is a $\ast$-algebra, in which all of the
algebraic operations on $\mathcal{M}$ have continuous extensions. Moreover,
every $A\in\mathcal{\hat{M}}$ has a polar decomposition $A=U\left\vert
A\right\vert $ where $U$ is a unitary in $\mathcal{M}$ and $\left\vert
A\right\vert =\left(  A^{\ast}A\right)  ^{1/2}$. A norm $\beta$ on
$\mathcal{M}$ is called a \emph{normalized unitarily invariant norm} on
$\mathcal{M}$ if, for every $T\in\mathcal{M}$ and all unitary operators
$U,V\in\mathcal{M}$,%
\[
\beta\left(  UTV\right)  =\beta\left(  T\right)  \text{.}%
\]
Examples of such norms are the $p$-norms, $1\leq p<\infty$, defined on
$\mathcal{M}$ by%
\[
\left\Vert T\right\Vert _{p}=\tau\left(  \left\vert T\right\vert ^{p}\right)
^{1/p},
\]
where $\left\vert T\right\vert =\left(  T^{\ast}T\right)  ^{1/2}$. We define
$L^{\beta}\left(  \mathcal{M},\tau\right)  $ to be the completion of
$\mathcal{M}$ with respect to $\beta.$ When $\beta\geq\left\Vert .\right\Vert
_{1}$, we get $L^{\beta}\left(  \mathcal{M},\tau\right)  \subset L^{1}\left(
\mathcal{M},\tau\right)  \subset\mathcal{\hat{M}}$. It was shown in \cite{HN}
that $\left(  L^{\beta}\left(  \mathcal{M},\tau\right)  ,\mathcal{\hat{M}%
}\right)  $ is a multiplier pair and $\mathcal{L}_{0}=\mathcal{R}%
_{0}=\mathcal{M}$ and $e=1$. Suppose $T\in\mathcal{M}$. We define $\varphi
_{T}\in L^{\beta}\left(  \mathcal{M},\tau\right)  ^{\#}$, by%
\[
\varphi_{T}\left(  A\right)  =\tau\left(  TA\right)  \text{.}%
\]
If we let $\alpha=\varphi_{1}$, we see that $1\otimes\alpha=\tau$ is tracial
on $\mathcal{M}$. Also, if $A\in L^{p}\left(  \mathcal{M},\tau\right)  \subset
L^{1}\left(  \mathcal{M},\tau\right)  $ with polar decomposition $U\left\vert
A\right\vert $, and if
\[
\left(  L_{D}^{\#}\alpha\right)  \left(  A\right)  =0
\]
for every $D\in\mathcal{M}$, then%
\[
0=L_{U^{\ast}}^{\#}\left(  A\right)  =\tau\left(  U^{\ast}A\right)
=\tau\left(  \left\vert A\right\vert \right)  .
\]
This implies $A=0.$ Thus $\left\{  L_{D}^{\#}\alpha:D\in\mathcal{L}%
_{0}\right\}  $ is weak* dense in $L^{\beta}\left(  \mathcal{M},\tau\right)
$. If we apply Theorem \ref{35}, we obtain the following result.

\begin{theorem}
\label{40}Let $\mathcal{M}$ be a von Neumann algebra with a faithful normal
tracial state. If $\beta\geq\left\Vert .\right\Vert _{1}$ is a normalized
unitarily invariant norm, then $\mathcal{M=}\left\{  L_{A}:A\in\mathcal{M}%
\right\}  $ is a maximal $1\otimes\varphi_{1}$-tracial algebra in $B\left(
L^{\beta}\left(  \mathcal{M},\tau\right)  \right)  .$
\end{theorem}

An example of a von Neumann algebra with a faithful normal tracial state is
$L^{\infty}\left(  \mu\right)  $ where $\left(  \Omega,\mu\right)  $ is a
probability space and $\tau:L^{\infty}\left(  \mu\right)  \rightarrow
\mathbb{C}$ is defined by%
\[
\tau\left(  f\right)  =\int_{\Omega}fd\mu\text{ .}%
\]
A normalized unitarily invariant norm $\beta$ on $L^{\infty}\left(
\mu\right)  $ is a \emph{normalized gauge norm}, i.e., a norm such that
$\beta\left(  1\right)  =1$ and $\beta\left(  f\right)  =\beta\left(
\left\vert f\right\vert \right)  $ for every $f\in L^{\infty}\left(
\mu\right)  $. If $\beta\geq\left\Vert .\right\Vert _{1}$, we have that
$L^{\beta}\left(  \mu\right)  $ is a multiplier pair and $\mathcal{L}%
_{0}=\mathcal{R}_{0}=L^{\infty}\left(  \mu\right)  $ is maximal abelian in
$B\left(  L^{\beta}\left(  \mu\right)  \right)  $. Thus Theorem \ref{10}
applies. A vector $f\in L^{\beta}\left(  \mu\right)  $ is cyclic for
$L^{\infty}\left(  \mu\right)  $ if and only if $f\neq0$ a.e. $\left(
\mu\right)  $.

\begin{corollary}
\label{45}Suppose $\left(  \Omega,\mu\right)  $ is a probability space,
$\beta\geq\left\Vert .\right\Vert _{1}$ is a normalized gauge norm on
$L^{\infty}\left(  \mu\right)  ,$ $f\in L^{\beta}\left(  \mu\right)  $ and
$\alpha\in L^{\beta}\left(  \mu\right)  $ satisfy
\[
f\neq0\text{ a.e. }\left(  \mu\right)  \text{ and }L^{\infty}\left(
\mu\right)  ^{\#}\alpha\text{ is weak* dense in }L^{\beta}\left(  \mu\right)
^{\#}.
\]
Then $L^{\infty}\left(  \mu\right)  $ is maximal $\left(  f\otimes
\alpha\right)  $-tracial.
\end{corollary}

We now consider an important problem. Suppose $\mathcal{M}$ is a von Neumann
algebra on a Hilbert space $\mathcal{H}$ and $e,f\in\mathcal{H}$ and
$\left\langle e,f\right\rangle =1.$ If $e\otimes f$ is tracial on
$\mathcal{M}$, when is $\mathcal{M}$ maximal $\left(  e\otimes f\right)  $-tracial?

We first note that if $E$ is an invariant subspace for $\mathcal{M}$, then
$E^{\perp}$ is also invariant, so Theorem \ref{30} implies that a necessary
condition for $\mathcal{M}$ to be maximal $\left(  e\otimes f\right)
$-tracial is that $e$ and $f$ be cyclic vectors for $\mathcal{M}$. The real
question is whether the converse is true.

\begin{problem}
\label{50}Suppose that $\mathcal{M}$ is a von Neumann algebra on a Hilbert
space $\mathcal{H}$ and $e,f$ are cyclic vectors for $\mathcal{M}$ with
$\left\langle e,f\right\rangle =1.$ If $\mathcal{M}$ is $\left(  e\otimes
f\right)  $-tracial, then is $\mathcal{M}$ maximal $\left(  e\otimes f\right)
$-tracial?
\end{problem}

In the case where $\mathcal{H}$ is separable, we will reduce the problem to
the case where $\mathcal{M}$ is a finite factor von Neumann algebra, and such
a factor has a unique norm continuous unital tracial functional.

Suppose $\mathcal{M}$ is a von Neumann algebra on a separable Hilbert space
$\mathcal{H}$ and $e,f$ are cyclic vectors for $\mathcal{M}$ with
$\left\langle e,f\right\rangle =1.$ Let
\begin{align*}
\mathcal{J}  &  =\left\{  A\in\mathcal{M}:\left(  e\otimes f\right)  \left(
AB\right)  =0\text{ for all }B\in\mathcal{M}\right\} \\
&  =\left\{  A\in\mathcal{M}:Ae\perp\mathcal{M}f\right\} \\
&  =\left\{  A\in\mathcal{M}:Ae=0\right\}  .
\end{align*}
Clearly, $\mathcal{J}$ is a WOT-closed linear space. Also, if $e\otimes f$ is
tracial on $\mathcal{M}$, then $\mathcal{J}$ is a two-sided ideal in
$\mathcal{M}$. Hence there is a projection $P$ in the center $\mathcal{Z}%
\left(  \mathcal{M}\right)  $ such that $\mathcal{J=}P\mathcal{M}$. Thus
$P\mathcal{M}e=\mathcal{M}Pe=$ $0$. Since $\mathcal{M}e$ is dense in
$\mathcal{H}$, it follows that $P=0$ and $\mathcal{J}=0.$ Similarly, if
$A\in\mathcal{M}$ and $Af=0$, then $A=0.$ Thus $e$ and $f$ are cyclic
separating vectors for $\mathcal{M}$.

If we write the central decomposition for $\mathcal{M}$, we get
\[
\mathcal{H}=\int_{\Omega}^{\oplus}\mathcal{H}_{\omega}d\mu\left(
\omega\right)  ,\text{ }\mathcal{M}=\int_{\Omega}^{\oplus}\mathcal{M}_{\omega
}d\mu\left(  \omega\right)  ,\text{ }%
\]%
\[
e=\int_{\Omega}^{\oplus}e_{\omega}d\mu\left(  \omega\right)  ,\text{ }%
f=\int_{\Omega}^{\oplus}f_{\omega}d\mu\left(  \omega\right)  \text{, and}%
\]
and
\[
\mathcal{Z}\left(  \mathcal{M}\right)  =\int_{\Omega}^{\oplus}%
\mathcal{\mathbb{C}}I_{\omega}d\mu\left(  \omega\right)  \text{ ,}%
\]
where each $\mathcal{M}_{\omega}$ is a factor von Neumann algebra. Since $e$
and $f$ are separating vectors for $\mathcal{M}$, we have $e_{\omega}\neq0$
and $f_{\omega}\neq0$ almost everywhere. We also get $e_{\omega}$ and
$f_{\omega}$ are cyclic and separating for $\mathcal{M}_{\omega}$ a.e.
$\left(  \mu\right)  $. We also have $e_{\omega}\otimes f_{\omega}$ is tracial
for almost every $\omega$. Since a factor von Neumann algebra $\mathcal{M}%
_{\omega}$ can have at most one nonzero continuous unital tracial state
$\tau_{\omega}$, we know that almost every $\mathcal{M}_{\omega}$ is a finite
factor and has a unique tracial functional $\tau_{\omega}$ of norm $1.$ We see
that $\left\langle e_{\omega},f_{\omega}\right\rangle \neq0$ a.e. $\left(
\mu\right)  $. Thus $\tau_{\omega}=\frac{1}{\left\langle e_{\omega},f_{\omega
}\right\rangle }e_{\omega}\otimes f_{\omega}$.

Suppose now that $\mathcal{M}\subset\mathcal{D}$ and $\mathcal{D}$ is $\left(
e\otimes f\right)  $-tracial. Suppose $P$ is a central projection in
$\mathcal{M}$. Then, for all $A,B\in\mathcal{M}$, we have%
\begin{align*}
\left(  e\otimes f\right)  \left(  A\left(  PT-TP\right)  B\right)   &
=\left(  e\otimes f\right)  \left(  APTB\right)  -\left(  e\otimes f\right)
\left(  ATPB\right) \\
&  =\left(  e\otimes f\right)  \left(  \left(  BAPT\right)  \right)  -\left(
e\otimes f\right)  \left(  \left(  BPAT\right)  \right) \\
&  =\left(  e\otimes f\right)  \left(  \left(  PBAT\right)  \right)  -\left(
e\otimes f\right)  \left(  \left(  PBAT\right)  \right) \\
&  =0.
\end{align*}
Thus
\[
\left\langle \left(  PT-TP\right)  \left(  Be\right)  ,\left(  A^{\ast
}f\right)  \right\rangle =0
\]
for every $A,B\in\mathcal{M}$. Since $e$ and $f$ are cyclic, we see that
$PT-TP=0$ for every projection $P\in\mathcal{Z}\left(  \mathcal{M}\right)  $.
Thus $T\in\mathcal{Z}\left(  \mathcal{M}\right)  ^{\prime}$. Hence we can
write $T=\int_{\Omega}^{\oplus}T_{\omega}d\mu\left(  \omega\right)  $. Since
the algebra generated by $T$ and $\mathcal{M}$ is $\left(  e\otimes f\right)
$-tracial, it follows that, for almost all $\omega\in\Omega$, the algebra
generated by $T_{\omega}$ and $\mathcal{M}_{\omega}$ is $\left(  e_{\omega
}\otimes f_{\omega}\right)  $-tracial. It now easily follows that
$\mathcal{M}$ is maximal $\left(  e\otimes f\right)  $-tracial if and only if
$\mathcal{M}_{\omega}$ is maximal $\left(  \frac{1}{\left\langle e_{\omega
},f_{\omega}\right\rangle }e_{\omega}\otimes f_{\omega}\right)  $-tracial a.e.
$\left(  \mu\right)  $. Hence when $\mathcal{H}$ is separable, this reduces
the problem to the case where $\mathcal{M}$ is a finite factor von Neumann algebra.

Suppose now that $\mathcal{M}$ is a finite factor and $\tau$ is the unique
continuous unital tracial functional on $\mathcal{M}$. Since $e$ and $f$ are
separating cyclic vectors, we can assume, via unitary equivalence \cite{Kad0},
that $\mathcal{H}=L^{2}\left(  \mathcal{M},\tau\right)  $ and $e,f\in
L^{2}\left(  \mathcal{M},\tau\right)  \subset L^{1}\left(  \mathcal{M}%
,\tau\right)  \subset\mathcal{\hat{M}}$ as above. We now have $e\otimes
f=\tau$ on $\mathcal{M}$. This means that, for every $A\in\mathcal{M}$,%
\[
\tau\left(  f^{\ast}Ae\right)  =\tau\left(  Aef^{\ast}\right)  =\tau\left(
A\right)  \text{.}%
\]
Thus
\[
\tau\left(  A\left(  1-ef^{\ast}\right)  \right)  =0.
\]
It follows $ef^{\ast}=1$ (where the multiplication is in $\mathcal{\hat{M}}$).
Since $\mathcal{M}$ is finite, $f^{\ast}e=1$. Thus $e,e^{-1}\in L^{2}\left(
\mathcal{M},\tau\right)  $ and $f^{\ast}=e^{-1},$ so $f=\left(  e^{-1}\right)
^{\ast}$. Thus $e\otimes f=e\otimes\left(  e^{-1}\right)  ^{\ast}$. Now the
question asks whether $\mathcal{M}$ is maximal $\left(  e\otimes\left(
e^{-1}\right)  ^{\ast}\right)  $-tracial.

If our question has an affirmative answer, it would imply that if
$\mathcal{M}\subset\mathcal{B}\left(  \mathcal{H}\right)  $ is a von Neumann
algebra and $S\in\mathcal{B}\left(  \mathcal{H}\right)  $ is invertible and
$\mathcal{A}=S^{-1}\mathcal{M}S$ and $\left(  e\otimes f\right)  $ is a unital
tracial functional on $\mathcal{A}$, then $\mathcal{A}$ is maximal $\left(
e\otimes f\right)  $-tracial.

There is a slight relationship between these ideas and a famous problem in
C*-algebra theory called the \emph{Kadison Similarity Problem} \cite{Kad},
which asks if every bounded unital homomorphism $\rho:\mathcal{A}\rightarrow
B\left(  \mathcal{H}\right)  $ from a C*-algebra $\mathcal{M}$ is similar to a
$\ast$-homomorphism. It is enough to prove the case when $\mathcal{A}$ is
separable. This is equivalent to the question when $\mathcal{A}$ is a von
Neumann algebra acting on a separable Hilbert space, and $\rho$ is weak*-weak*
continuous \cite{Kad}. It follows from \cite{Kad} that the problem reduces to
the case when $\mathcal{A}$ is a $II_{1}$ factor von Neumann algebra on a
separable Hilbert space. It follows from a theorem of R. Kadison, that when
$\mathcal{A}$ is a $II_{1}$ factor and $\rho:\mathcal{A}\rightarrow B\left(
\mathcal{H}\right)  $ is weak*-weak* continuous, that $\rho\left(
\mathcal{A}\right)  $ is similar to a von Neumann algebra.

\textbf{Question:} Suppose $\mathcal{A}$ is a $II_{1}$ factor von Neumann
algebra, $\rho:\mathcal{A}\rightarrow B\left(  \mathcal{H}\right)  $ is a
weak*-weak* continuous bounded $\ast$-homomorphism, and $e,f\in\mathcal{H}$
satisfy $\left\langle e,f\right\rangle =1,$ $\rho\left(  \mathcal{A}\right)  $
is $\left(  e\otimes f\right)  $-tracial, and $e$ is cyclic for $\rho\left(
\mathcal{A}\right)  $ and $f$ is cyclic for $\rho\left(  \mathcal{A}\right)
^{\ast}$, must $\mathcal{\rho}\left(  \mathcal{A}\right)  $ be maximal
$\left(  e\otimes f\right)  $-tracial? Is the lattice of invariant subspaces
of $\rho\left(  \mathcal{A}\right)  $ complemented?

\section{Maximal Tracial Ultraproducts}

The theory of ultraproducts has played a fundamental role in many areas of
mathematics, logic, Banach spaces, von Neumann algebras and C*-algebras. For
von Neumann algebras, the tracial ultraproducts have been extremely important.
We define a new type of ultraproduct. Suppose $\left\{  \mathcal{A}_{i}:i\in
I\right\}  $ is an infinite collection of unital C*-algebras. Also suppose
$\alpha$ is a nontrivial ultrafilter on $I$. We let $\mathcal{A}=%
{\displaystyle\prod_{i\in I}}
\mathcal{A}_{i}$ be the C*-direct product, and we define
\[
\mathcal{J}=\left\{  \left\{  a_{i}\right\}  \in\mathcal{A}:\lim
_{i\rightarrow\alpha}\left\Vert a_{i}\right\Vert =0\right\}  \text{.}%
\]
Then $\mathcal{J}$ is a closed two-sided ideal in $\mathcal{A}$, and
$\mathcal{A}/\mathcal{J}$ is called the\emph{ ultraproduct }of the C*-algebras
$\mathcal{A}_{i}$, and is denoted by%
\[%
{\displaystyle\prod_{i\in I}^{\alpha}}
\mathcal{A}_{i}\text{ .}%
\]
Next suppose for each $i\in I$ we have $\tau_{i}$ is a tracial state on
$\mathcal{A}_{i}$. If we define
\[
\mathcal{J}_{\alpha}=\left\{  \left\{  a_{i}\right\}  \in\mathcal{A}%
:\lim_{i\rightarrow\alpha}\tau_{i}\left(  a_{i}^{\ast}a_{i}\right)
=0\right\}  \text{,}%
\]
then $\mathcal{J}_{\alpha}$ is a closed two-sided ideal in $\mathcal{A}$, and
$\mathcal{A}/\mathcal{J}_{\alpha}$ is called the\emph{ tracial ultraproduct
}of the C*-algebras $\mathcal{A}_{i}$, and is denoted by%
\[%
{\displaystyle\prod_{i\in I}^{\alpha}}
\left(  \mathcal{A}_{i},\tau_{i}\right)  \text{.}%
\]
S. Sakai \cite{Sakai} proved that if each $\mathcal{A}_{i}$ is a factor von
Neumann algebra, then $%
{\displaystyle\prod_{i\in I}^{\alpha}}
\left(  \mathcal{A}_{i},\tau_{i}\right)  $ is a factor von Neumann algebra.
More recently, the first author and W. H. Li \cite{HadLi} proved that every
tracial ultraproduct of C*-algebras is a von Neumann algebra.

Another way to view a tracial ultraproduct, is by defining a tracial state
$\tau_{\alpha}$ on $%
{\displaystyle\prod_{i\in I}}
\mathcal{A}_{i}$ by%
\[
\tau_{\alpha}\left(  \left\{  a_{i}\right\}  \right)  =\lim_{i\rightarrow
\alpha}\tau_{i}\left(  a_{i}\right)  .
\]
Then $\mathcal{J}_{\alpha}=\left\{  a\in%
{\displaystyle\prod_{i\in I}}
\mathcal{A}_{i}:\tau_{\alpha}\left(  a^{\ast}a\right)  =0\right\}  $, which is
the first part of the GNS construction for $\tau_{\alpha}$.

We now consider the case where, for each $i\in I$, $\varphi_{i}$ is a state on
$\mathcal{A}_{i}$. We define $\varphi_{\alpha}$ on $%
{\displaystyle\prod_{i\in I}}
\mathcal{A}_{i}$ by%
\[
\varphi_{\alpha}\left(  \left\{  a_{i}\right\}  \right)  =\lim_{i\rightarrow
\alpha}\varphi_{i}\left(  a_{i}\right)  .
\]
Let $\mathcal{M}\subset\mathcal{A}$ be a maximal tracial C*-algebra with
respect to $\varphi$. We define the maximal tracial ultraproduct with respect
to $\mathcal{M}$, denoted by%
\[%
{\displaystyle\prod_{\mathcal{M}}^{\alpha}}
\left(  \mathcal{A}_{i},\varphi_{i}\right)  ,
\]
to be $\mathcal{M}/\mathcal{J}_{\alpha}\left(  \mathcal{M}\right)  $, where
$\mathcal{J}\left(  \alpha,\mathcal{M}\right)  =\left\{  a\in\mathcal{M}%
:\varphi_{\alpha}\left(  a^{\ast}a\right)  =0\right\}  $. Since $\varphi
_{\alpha}|_{\mathcal{M}}$ is tracial, $\varphi_{\alpha}$ induces a tracial
state $\hat{\varphi}_{\alpha}.$

If each $\varphi_{i}$ is tracial, then $\mathcal{A}$ is $\varphi$-tracial and
$%
{\displaystyle\prod_{\mathcal{A}}^{\alpha}}
\left(  \mathcal{A}_{i},\varphi_{i}\right)  $ is the usual tracial
ultraproduct. If $\varphi$ is tracial on $\mathcal{A}$, then $%
{\displaystyle\prod_{\mathcal{A}}^{\alpha}}
\left(  \mathcal{A}_{i},\varphi_{i}\right)  $ is the tracial ultraproduct
defined by H. Ando and E. Kirchberg \cite{AndoKir}.

One of the biggest open problems in the theory of von Neumann algebras is
\emph{Connes' Embedding Problem} \cite{Connes}, which asks if every von
Neumann algebra $\mathcal{R}$ acting on a separable Hilbert space with a
faithful normal tracial state $\tau$ can be tracially embedded into a tracial
ultraproduct $%
{\displaystyle\prod_{i\in I}^{\alpha}}
\left(  \mathcal{A}_{i},\tau_{i}\right)  $ with each $\mathcal{A}_{i}$
finite-demensional. This means there is a unital $\ast$-homomorphism
$\pi:\mathcal{R}\rightarrow%
{\displaystyle\prod_{i\in I}^{\alpha}}
\left(  \mathcal{A}_{i},\tau_{i}\right)  $ such that $\tau=\hat{\tau}_{\alpha
}\circ\pi$.

The analogue of Connes' embedding problem has an affirmative answer in the
setting of maximal tracial ultraproducts.

\begin{theorem}
\label{60}Every von Neumann algebra $\mathcal{R}$ acting on a separable
Hilbert space with a faithful normal tracial state $\tau$ can be tracially
embedded into a tracial ultraproduct of finite-dimensional algebras.
\end{theorem}

\begin{proof}
We can assume that $\mathcal{R}\subset B\left(  \mathcal{H}\right)  $ for some
separable infinite-dimensional Hilbert space $\mathcal{H}$ and that there is a
unit vector $f\in\mathcal{H}$ such that, for every $T\in\mathcal{R}$,
\[
\tau\left(  T\right)  =\left\langle Tf,f\right\rangle .
\]
We can also assume that $f$ is a cyclic vector for $\mathcal{R}$. Let
$E=\left\{  e_{1},e_{2},\ldots\right\}  $ be an orthonormal basis for
$\mathcal{H}$ such that $e_{1}=f$. For each positive integer $n$, let
$P_{n}\ $be the orthogonal projection onto the linear span of $\left\{
e_{1},\ldots,e_{n}\right\}  $. Let $\mathcal{A}_{n}=P_{n}B\left(  H\right)
P_{n}$ and define a state $\varphi_{n}:\mathcal{A}_{n}\rightarrow\mathbb{C}$
by
\[
\varphi_{n}\left(  A\right)  =\left\langle Af,f\right\rangle =\left\langle
Ae_{1},e_{1}\right\rangle .
\]
Let $\alpha$ be any free ultrafilter on $\mathbb{N}$. Let
\[
\mathcal{M}_{0}=\left\{  \left\{  A_{n}\right\}  \in%
{\displaystyle\prod_{n\in\mathbb{N}}}
\mathcal{A}_{n}:\left\{  A_{n}\right\}  \text{ is *SOT convergent, }%
\lim_{n\rightarrow\infty}A_{n}\in\mathcal{R}\right\}  \text{.}%
\]
Since $\tau$ is a trace on $\mathcal{R}$, for $S=\left\{  A_{n}\right\}
,T=\left\{  B_{n}\right\}  \in\mathcal{M}_{0}$ with $A_{n}\rightarrow
A\in\mathcal{R}$ and $B_{n}\rightarrow B\in\mathcal{R}$ $\left(
\text{*SOT}\right)  $, we have%
\[
A_{n}B_{n}\rightarrow AB\text{ and }B_{n}A_{n}\rightarrow BA\text{ }\left(
\text{*SOT}\right)  .
\]
Thus
\[
\varphi_{\alpha}\left(  ST\right)  =\tau\left(  AB\right)  =\tau\left(
BA\right)  =\varphi_{\alpha}\left(  TS\right)  .
\]
Hence $\varphi_{\alpha}$ is tracial for the C*-algebra $\mathcal{M}_{0}$. Then
there is a maximal $\varphi_{\alpha}$-tracial C*-algebra $\mathcal{M}\subset%
{\displaystyle\prod_{n\in\mathbb{N}}}
\mathcal{A}_{n}.$

Let $\eta:\mathcal{M\rightarrow M}/\left(  \mathcal{M}\cap\mathcal{J}_{\alpha
}\right)  =%
{\displaystyle\prod_{\mathcal{M}}^{\alpha}}
\left(  \mathcal{A}_{n},\varphi_{n}\right)  $ be the quotient map and define
$\pi:\mathcal{R}\rightarrow%
{\displaystyle\prod_{\mathcal{M}}^{\alpha}}
\left(  \mathcal{A}_{n},\varphi_{n}\right)  $ by%
\[
\gamma\left(  A\right)  =\eta\left(  \left\{  P_{n}AP_{n}\right\}  \right)  .
\]
We easily see that, for every $A\in\mathcal{R}$,%
\[
\hat{\varphi}_{\alpha}\left(  \pi\left(  A\right)  \right)  =\lim_{\alpha
}\left\langle P_{n}AP_{n}f,f\right\rangle =\left\langle Af,f\right\rangle
=\tau\left(  A\right)  .
\]
Thus $\hat{\varphi}_{\alpha}\circ\pi=\tau$. Since $\tau$ is faithful on
$\mathcal{R}$, $\pi$ is an embedding.
\end{proof}

\end{document}